\documentclass{amsart}
\usepackage{amssymb, amsmath}
\usepackage[dvips]{graphicx}
\usepackage{amsfonts}
\usepackage{latexsym}
\usepackage{color}
\newtheorem{theorem}{Theorem}	
\newtheorem{lemma}{Lemma}[section]		
\newtheorem{corollary}{Corollary}		
\newtheorem{proposition}{Proposition}

\begin{document}
\title[Behavior of convex integrand at apex of its Wulff shape]
{Behavior of convex integrand at apex of its Wulff shape}
\author{
Huhe Han}
\address{College of Science, Northwest Agriculture and Forestry University, China}
\email{han-huhe@nwafu.edu.cn}
\subjclass[2020]{52A20, 52A55}
\keywords{Wulff shape, convex integrand, apex, approximation,  
constant width, Hausdorff distance}
\begin{abstract}
Let $\gamma: S^n\to \mathbb{R}_+$ be a convex integrand and $\mathcal{W}_\gamma$ be the Wulff shape of $\gamma$. 
Apex point naturally arise in non-smooth Wulff shape, in particular, vertex of convex polytope.
In this paper, 
we study the behavior of convex integrand around apex point of its Wulff shape.
We prove that 
$\gamma(P)$ is locally  maximum, and $\mathbb{R}_+ P\cap \partial \mathcal{W}_\gamma$ is an apex point of 
$\mathcal{W}_\gamma$ 
if and only if the graph of $\gamma$ around the apex point is a pice of sphere.
As an application of the proof of this result, we prove that 
for any spherical convex body $C$ of constant width $\tau>\pi/2$,
 there exists a sequence $\{C_i\}_{i=1}^\infty$ of convex bides of constant width $\tau$, 
whose boundary consists only of arcs of circles of radius $\tau-\frac{\pi}{2}$ and great circle segments such that 
$\lim_{i\to \infty}C_i=C$
with respect to the Hausdorff distance.
\end{abstract}
\maketitle

\section{Introduction}
Let $S^n$ and $\mathbb{R}_+$ be $n$ dimensional unit sphere in $\mathbb{R}^{n+1}$ and the set consisting of positive real numbers respectively. 
Let $\gamma:S^n\to \mathbb{R}_+$ be a continuous function. 
The {\it Wulff shape}  associated with support function $\gamma$, denoted by $\mathcal{W}_\gamma$, is the following set:
\[
\bigcap_{\theta\in S^n}\Gamma_{\gamma,\theta}.
\]
Here $\Gamma_{\gamma,\theta}$ is the half space 
\[
\{x\in\mathbb{R}^{n+1}|x\cdot\theta\leq \gamma(\theta)\},
\] 
where the dot in the center stands for the dot product of two vectors $x, \theta$ of $\mathbb{R}^{n+1}$.
By definition,  Wulff shape $\mathcal{W}_\gamma$ is compact, convex and the origin of $\mathbb{R}^{n+1}$ is an interior point of $\mathcal{W}_\gamma$.
The Wulff shape was first introduced by G.Wulff in \cite{wulff}, is known as a geometric model of a crystal at equilibrium (see for instance \cite {crystalbook,taylor}). 
Let $\mbox{inv}: \mathbb{R}^{n+1}-\{\bf 0\}\to\mathbb{R}^{n+1}-\{\bf 0\}$ be the 
{\it inversion} with respect to the origin {\bf 0} of $\mathbb{R}^{n+1}$,  is defined by
\[
\mbox{inv}(\theta, r)=(-\theta, \frac{1}{r}),
\]
where $(\theta, r)$  
is the polar plot expression for a point of $\mathbb{R}^{n+1}-{\bf 0}$. 
For given Wulff shape there are many support functions in general. 
It is known that for any continuous funtios $\gamma_1, \gamma_2: S^n\to \mathbb{R}_+$, 
the equality $\mathcal{W}_{\gamma_1}=\mathcal{W}_{\gamma_2}$ holds if and only if
the convex hull of $\{\mbox{inv}(\theta, \gamma_1(\theta))\mid \theta\in S^n\}$ 
 equaling
the convex hull of 
$\{
\mbox{inv}(\theta, \gamma_2(\theta))\mid \theta\in S^n
\}
$
exactly(\cite{nishimurasakemi2, taylor}).
A support function $\gamma$ is said to be {\it convex integrand} of $\mathcal{W}_\gamma$ if the set 
\[
\left\{
\mbox{inv}(\theta, \gamma(\theta))\mid \theta\in S^n
\right\}
\]
is the boundary of a convex body of $\mathbb{R}^{n+1}$. 
The definition of convex integrand was firstly introduced by J.E. Taylor in \cite{taylor}, 
and it plays a key role for studying Wulff shapes. 
For example, 
it is known that Wulff shape $\mathcal{W}$ is strictly convex if and only if its convex integrand
$\gamma$ is of class $C^1$ (\cite{hnams}); 
Wulff shape $\mathcal{W}$ is uniformly convex if and only if its convex
integrand $\gamma$ is of class $C^{1,1}$ (\cite{morgan91}).
\par
We say a boundary point $P$ of a plane convex body $C$ is {\it smooth} if exactly one line supports  $C$ at $P$. 
Let $Gr(n+1,2)^P$ be a Grassmannian subspace $Gr(n+1,2)$ of $\mathbb{R}^{n+1}$ contains $P$.
We say a boundary point $P$ of Wulff shape $\mathcal{W}_\gamma\subset \mathbb{R}^{n+1}$ is an {\it apex} point if 
$P$ is a non-smooth point of 
$\mathcal{W}_\gamma\cap Gr(n+1,2)^P$ for any $Gr(n+1,2)^P$.  
\par
In this paper, we study the behavior of a convex integrand around an 
apex point of its Wulff shape.
\begin{theorem}\label{maintheorem}
Let $\gamma: S^n\to \mathbb{R}_+$ be a convex integrand, and 
let $P$ be an element of $S^n$ and 
Set $B_s^n(P,\delta)=\{Q\in S^n\mid \arccos(P\cdot Q)<\delta\}$.
Then following three statements are equivalent:
\begin{enumerate}
\item there exists a positive number $\delta$ such that 
$\gamma(Q)=\gamma(P)P\cdot Q$ for any point $Q\in B_s^n(P, \delta)$;
\item $\gamma(P)$ is locally  maximum, and 
$\mathbb{R}_+ P\cap \partial \mathcal{W}_\gamma$ is an apex point of 
$\partial\mathcal{W}_\gamma$;
\item there exists a positive number $\delta$ such that 
\[
\{(-Q, \frac{1}{\gamma(Q)})\mid Q\in B_s^n(P,\delta)\}
\] is a pice of $\mathcal{H}^n$, where 
$\mathcal{H}^n$ is the $n-$dimensional affine hyperplane through $\frac{1}{\gamma(P)}(-P),$ and has a normal vector as $P$.
\end{enumerate}
\end{theorem}
Since 
\[
||\gamma(Q)Q-\frac{1}{2}\gamma(P)P||
=\left((\gamma(Q)Q-\frac{1}{2}\gamma(P)P)\cdot (\gamma(Q)Q-\frac{1}{2}\gamma(P)P)\right)^{\frac{1}{2}}
=\frac{1}{2}\gamma(P),
\]
by Assertion (1) of Theorem \ref{maintheorem}, the graph of the convex integrand $\gamma$ around an apex point $\mathbb{R}_+ P\cap \partial \mathcal{W}_\gamma$ of its Wulff shape is 
a pice of sphere with center $\frac{1}{2}\gamma(P)P$ and radius $\frac{1}{2}\gamma(P)$.
Here, $||\cdot ||$ denotes the standard Euclidean norm.
\par
\bigskip
As an application of the proof of Theorem \ref{maintheorem}, we have the following:
\begin{theorem}\label{theorem2}
For any spherical convex body $\widetilde{C}\subset S^2$ of constant width $\tau>\pi/2$, 
and for any $\varepsilon>0$ 
there exists a body $\widetilde{C}_\varepsilon$ of constant width $\tau$ 
whose boundary consists only of arcs of circles of radius $\tau-\frac{\pi}{2}$ and great circle segments, 
such that 
\[
h(\widetilde{C}, \widetilde{C}_\varepsilon)\leq\varepsilon,
\] 
where $h$ is the Hausdorff distance.
\end{theorem} 
Theorem \ref{theorem2} is a counterpart of the result of \cite{Lassak22-1}, which says that 
$\lq\lq$ For any spherical convex body $\widetilde{C}\subset S^2$ of constant width $\tau<\pi/2$, 
and for arbitrary $\varepsilon>0$ 
there exists a body $\widetilde{C}_\varepsilon\subset S^2$ of constant width $\tau$ 
whose boundary consists only of arcs of circles of radius $\tau$ 
such that 
the Hausdorff distance between $\widetilde{C}$ and $\widetilde{C}_\varepsilon$ is at most $\varepsilon $
(Theorem \ref{theorem3}). 
In Euclidean plane, following \cite{banchoffgiblin94}, a {\it PC curve}, or piecewise circular curve, is given by a finite sequence of circular arcs or line segments, with the endpoint of one arc coinciding with the beginning point of the next. More details on $PC$ curves, see for instance \cite{banchoffgiblin94, banchoffgiblin93}. 
Since $\lq\lq$straight line$"$ on $S^2$ is great circle, from Theorem 2, 
one obtains that any spherical convex body of constant width $\tau>\pi/2$  can be approximated by a body of constant width $\tau$ whose boundary is a spherical $PC$ curve.
\par
\bigskip
This paper is organized as follows. 
 In Section 2, preliminaries are given.
The proofs of Theorem \ref{maintheorem} and Theorem \ref{theorem2} are given in Section 3, Section 4 respectively.

\section{Premilinaries}
\subsection{An equivalent definition of Wulff shape}
In \cite{nishimurasakemi2}, T.\ Nishimura and Y.\ Sakemi presented 
an equivalent definition of Wulff shape as the composition of the following mappings.
\par (1)
Let $Id: \mathbb{R}^{n+1}\to \mathbb{R}^{n+1}\times \{1\}\subset \mathbb{R}^{n+2}$
be the mapping defined by
\[
Id(x)=(x,1).
\]
\par (2)
Let $N=(0,\ldots,0,1)\in \mathbb{R}^{n+2}$ denote the north pole of $S^{n+1}$,
and let $S_{N,+}^{n+1}$ denote the north open hemisphere of $S^{n+1}$, that is,
\[
S_{N,+}^{n+1}=\{\widetilde{Q}\in S^{n+1}| N\cdot \widetilde{Q}>0\}.
\]
The {\it central projection relative to} $N$, denoted by
$\alpha_N: S_{N,+}^{n+1}\to \mathbb{R}^{n+1}\times \{1\}$, is defined by
\[
\alpha_N\left(P_1, \ldots, P_{n+1}, P_{n+2}\right)
=
\left(\frac{P_1}{P_{n+2}}, \ldots, \frac{P_{n+1}}{P_{n+2}}, 1\right).
\]
\par (3)
The {\it spherical blow-up} (with respect to $N$) $\Psi_N:S^{n+1}-\{\pm N\}\to {\color{black}S_{N, +}^{n+1}}$, 
is defined by 
\[
\Psi_N(\widetilde{P})=\frac{1}{\sqrt{1-(N\cdot \widetilde{P})^2}}(N-(N\cdot \widetilde{P})\widetilde{P}).    
\]
The mapping $\Psi_N$, which was first introduced in \cite{nishimura}. 
By definition, it is obvious that  
\[
|\Psi_N(\widetilde{P})\widetilde{P}|=\pi/2\]
for any point $\widetilde{P}\in S^{n+1}-\{\pm N\}$.
Using $\Psi_N$, the inversion 
$\mbox{inv}:\mathbb{R}^{n+1}-{\bf 0}\to\mathbb{R}^{n+1}-{\bf 0}$ 
can be characterized as follows:
\begin{proposition}[\cite{hnams}]\label{psi}
\[
\mbox{{\rm inv}}=\mbox{Id}^{-1}\circ \alpha_N \circ \Psi_N\circ \alpha_N^{-1} \circ{\mbox Id}.
\]
\end{proposition}
\par 
(4) 
For any point $\widetilde{P} \in S^{n+1}$, let $H(\widetilde{P})$
be the hemisphere centered at $\widetilde{P}$,
\[
H(\widetilde{P})=\{\widetilde{Q}\in S^{n+1}\mid \widetilde{P}\cdot \widetilde{Q}\geq 0\},
\]
where the dot in the center stands for the scalar product of two vectors 
$\widetilde{P}, \widetilde{Q}\in \mathbb{R}^{n+2}$.     
For any {\color{black}non-empty} subset $\widetilde{W}\subset S^{n+1}$, the {\it spherical polar set of $\widetilde{W}$}, denoted by 
$\widetilde{W}^\circ$, is defined as follows: 
\[
\widetilde{W}^\circ = \bigcap_{\widetilde{P}\in \widetilde{W}}H(\widetilde{P}).
\]   
\begin{proposition}[\cite{nishimurasakemi2}]\label{sphericalmethod}
Let $\gamma: S^{n}\to \mathbb{R}_+$ be a continuous function.    
Set 
\[
\mbox{\rm graph}(\gamma)=\{(\theta, \gamma(\theta))\in \mathbb{R}^{n+1}-\{0\}\; |\; \theta\in S^n\},
\]
where 
$(\theta, \gamma(\theta))$ is the polar plot expression for a point of $\mathbb{R}^{n+1}-\{0\}$.   
Then, $\mathcal{W}_\gamma$ {\color{black}is} characterized as follows:   
\[
\mathcal{W}_\gamma = 
Id^{-1}\circ \alpha_{{}_{N}}\left(\left(\Psi_N\circ \alpha_{{}_{N}}^{-1}\circ 
Id\left(\mbox{\rm graph}(\gamma)\right)\right)^\circ\right).
\]
\end{proposition}
\subsection{Spherical polar sets} 
\begin{lemma}[\cite{nishimurasakemi2}]\label{lemmainclusion}
Let $\widetilde{X}, \widetilde{Y}$ be subsets of $S^{n+1}$. Suppose that the inclusion 
$\widetilde{X}\subset \widetilde{Y}$ holds.
Then, the inclusion $\widetilde{Y}^\circ \subset \widetilde{X}^\circ$ holds.
\end{lemma}
A subset $\widetilde{X}\subset S^{n+1}$ is said to be {\it hmispherical} if there exists a
point $\widetilde{P}$ of  $S^{n+1}$ such that $H(\widetilde{P})\cap \widetilde{X}=\emptyset$.
Let $\widetilde{C}$ be a hemispherical subset of $S^{n+1}$. 
The {\it spherical convex hull of} $\widetilde{C}$, denoted by $\mbox{s-conv}(\widetilde{C})$, is defined by
\[
\mbox{s-conv}(\widetilde{C}) =\left\{
\frac{\sum_{i=1}^k t_i\widetilde{P}_i}{\mid\mid\sum_{i=1}^k t_i\widetilde{P}_k\mid\mid} |
\sum_{i=1}^kt_i=1, t_i\geq 0, \widetilde{P}_1,\dots, \widetilde{P}_k\in \widetilde{C}, k\in \mathbb{N}
\right\}.
\]
\begin{proposition}[\cite{nishimurasakemi2}]\label{dualconvex}
For any closed hemispherical subset $\widetilde{X}\subset S^{n+1}$, 
the following equality holds:
\[
\mbox{{\rm s-conv}}(\widetilde{X})=(\mbox{{\rm s-conv}}(\widetilde{X}))^{\circ\circ}. 
\]
\end{proposition}
For any $\widetilde{P}, \widetilde{Q}\in S^{n+1} (\widetilde{P}\neq -\widetilde{Q})$, the arc connecting $\widetilde{P}$ and $\widetilde{Q}$, denoted by $\widetilde{P}\widetilde{Q}$,  is defined by
the shorter part of the great circle containing these points, namely,
\[
\widetilde{P}\widetilde{Q}=\left\{
\frac{(1-t)\widetilde{P}+t\widetilde{Q}}{||(1-t)\widetilde{P}+t\widetilde{Q}||}\in S^{n+1} \mid t\in [0,1]
\right\}.
\]
 The spherical distance $|\widetilde{P}\widetilde{Q}|$ of $\widetilde{P}$ and $\widetilde{Q}$ can be given by
$\arccos(\widetilde{P}\cdot \widetilde{Q})$.
A hemispherical subset $\widetilde{X}\subset S^{n+1}$ is said to be {\it spherical convex}
if the arc $\widetilde{P}\widetilde{Q}$ is a subset of $\widetilde{X}$ for any $\widetilde{P}, \widetilde{Q}\in \widetilde{X}$. 
A spherical convex set $\widetilde{C}$ is said to be {\it spherical convex body}
if $\widetilde{C}$ contains an interior point.
\begin{proposition}[\cite{hnsm}]\label{dualisometric}
Let $\widetilde{P}\in S^{n+1}$ and let $\widetilde{W}_1, \widetilde{W}_2$ be spherical convex bodies contained in $H(\widetilde{P})$. 
Suppose that $\widetilde{P}$ is an interior point of $\widetilde{W}_1, \widetilde{W}_2$.
Then the following equality holds:
\[
h(\widetilde{W}_1, \widetilde{W}_2)=h(\widetilde{W}_1^\circ, \widetilde{W}_2^\circ),
\]
where $h$ is the Hausdorff distance.
\end{proposition}
\subsection{Width of spherical convex bodies}
Denote by $\partial \widetilde{C}$ the boundary of spherical convex body $\widetilde{C}$.
Let $\widetilde{P}$ be a point of $\partial\widetilde{C}$. 
If 
\[
\widetilde{C}\subset H(\widetilde{Q})\ \mbox{and}\ \widetilde{P}\in \widetilde{C}\cap \partial H(\widetilde{Q}),
\] then we say hemisphere $H(\widetilde{Q})$ {\it supports} $\widetilde{C}$ at $\widetilde{P}$, and $H(\widetilde{Q})$ is a {\it supporting hemisphere} of $\widetilde{C}$ at $\widetilde{P}$. 
\begin{lemma}[\cite{hwam}]\label{boundarysupport}
Let $\widetilde{C}$ be a spherical convex body in $S^{n+1}$. 
For any point $\widetilde{P}$ of the boundary of $\widetilde{C}^\circ$,
the center of the supporting hemisphere of $\widetilde{C}^\circ$ at $\widetilde{P}$ is 
a boundary point of $\widetilde{C}$.
\end{lemma}
\begin{lemma}\label{lemmadualboundary}
Let $\widetilde{C}$ be a spherical convex body in $S^{n+1}$. 
Let $\widetilde{P}$ be a boundary point of $\widetilde{C}$. 
Then $H(\widetilde{P})$ is a supporting hemisphere of $\widetilde{C}^\circ$.
\end{lemma}
\begin{proof}
By definition of spherical polar set, $\widetilde{C}^\circ$ is a subset of $H(\widetilde{P})$.
Let $H(\widetilde{Q})$ be a supporting hemisphere of $\widetilde{C}$ at $H(\widetilde{P})$.
Then by Lemma \ref{boundarysupport}, $\widetilde{Q}$ is a boundary point of $\widetilde{C}^\circ$ and $|\widetilde{P}\widetilde{Q}|=\pi/2$. Thus, it follows that
\[
\widetilde{Q}\in \partial \widetilde{C}^\circ\cap H(\widetilde{P}).
\]
Therefore, $H(\widetilde{P})$ supports $\widetilde{C}^\circ$ at $\widetilde{Q}$.
\end{proof}
By the proof of Lemma \ref{lemmadualboundary}, we have the following:
\begin{lemma}
Let $C$ be a spherical convex body in $S^{n+1}$. The
following two assertions are equivalent:
\begin{enumerate}
\item the hemisphere $H(\widetilde{P})$ supports $\widetilde{C}^\circ$ at $\widetilde{Q}$;
\item the hemisphere $H(\widetilde{Q})$ supports $\widetilde{C}$ at $\widetilde{P}$.
\end{enumerate}
\end{lemma}
A spherical convex body $\widetilde{C}$ is said to be {\it smooth} if  for any boundary point $\widetilde{P}$ of $\widetilde{C}$,
 exactly one hemisphere supports $\widetilde{C}$ at $\widetilde{P}$.
If hemispheres $H(\widetilde{P})$ and $H(\widetilde{Q})$ of $S^{n+1}$ are different and $\widetilde{P}\neq -\widetilde{Q}$, 
then the intersection $H(\widetilde{P})\cap H(\widetilde{Q})$ is called a {\it lune} of $S^{n+1}$. 
The {\it thickness of lune} $H(\widetilde{P})\cap H(\widetilde{Q})$, denoted by $\Delta (H(\widetilde{P})\cap H(\widetilde{Q}))$, is given by 
\[
\Delta (H(\widetilde{P})\cap H(\widetilde{Q}))=\pi-|\widetilde{P}\widetilde{Q}|. 
\]
If $H(\widetilde{P})$ is a supporting hemisphere of a spherical convex body $\widetilde{C}$, 
{\it the width of} $\widetilde{C}$ with respect to $H(\widetilde{P})$, denoted by $\mbox{width}_{H(\widetilde{P})}(\widetilde{C})$, is defined by 
\[
\mbox{width}_{H(\widetilde{P})} (\widetilde{C}) =\mbox{min}\{\Delta(H(\widetilde{P})\cap H(\widetilde{Q}))| H(\widetilde{Q})\ \mbox{supports}\ \widetilde{C}\}.
\] We say that a spherical convex body $\widetilde{C}$ is {\it of constant width},
if all widths of $\widetilde{C}$ with respect to any supporting hemispheres are equal. 
Following \cite{Lassak15}, 
we define {\it the thickness} of a convex body $\widetilde{C}\subset S^n$, denoted by $\Delta(\widetilde{C})$,  as 
the minimum of $\mbox{width}_{H(\widetilde{P})} (\widetilde{C})$ over all supporting hemispheres $H(\widetilde{P})$ of $\widetilde{C}$.
That is 
\[
\Delta(\widetilde{C}) = \mbox{min}\{\mbox{width}_{\widetilde{K}}(\widetilde{C})| \widetilde{K}\ \mbox{is\ a\ supporting\ hemisphere\ of\ }\widetilde{C}\}.
\]
More details on width, thickness of spherical convex bodies, see for instance 
\cite{perimeter20,  hnjmsj, Lassak22, Musidlak20}.
\par
\bigskip
{\color{black}
Theorem \ref{theorem2} is a counterpart of the following result of \cite{Lassak22-1}.}
\begin{theorem}[\cite{Lassak22-1}]\label{theorem3}
For any spherical convex body $C\subset S^2$ of constant width $\tau<\pi/2$, 
and for any $\varepsilon>0$ 
there exists a body $C_\varepsilon$ of constant width $\Delta(C)=\Delta(C_\varepsilon)$ 
whose boundary consists only of arcs of circles of radius $\Delta(C)$ 
such that 
\[
h(C, C_\varepsilon)\leq \varepsilon,
\] 
where $h(C_1, C_2)$ means the Hausdorff distance between $C_1$ and $C_2$ .
\end{theorem}
\section{Proof of Theorem \ref{maintheorem}}
Let $H(\widetilde{Q}_1), H(\widetilde{Q}_2), H(\widetilde{Q}_3)$ be different supporting hemispheres of 
$\widetilde{C}\subset S^2$. 
Following \cite{LM18-1}, we write $\prec \widetilde{Q_1}\widetilde{Q_2}\widetilde{Q_3}$,
 if $\widetilde{Q}_1, \widetilde{Q}_2, \widetilde{Q}_3$ are in this order on the boundary of 
 $\widetilde{C}^\circ$
when viewed from the inside of $\widetilde{C}^\circ$. 
The supporting hemispheres $H(\widetilde{Q}_1)$ and $H(\widetilde{Q}_2)$ of $\widetilde{C}$ are said to be {\it extreme supporting hemispheres of} $\widetilde{C}$ 
if 
$H(\widetilde{Q}_1)$ and $H(\widetilde{Q}_2)$ 
be hemispheres supporting 
$\widetilde{C}$ (at $\widetilde{P}$) such that 
$\prec \widetilde{Q}_1\widetilde{Q}_3\widetilde{Q}_2$
for every hemisphere $H(\widetilde{Q_3})$ supporting $\widetilde{C}$ (at $\widetilde{P}$). 
Let  $c_1, c_2$ be centers of the two semi-great circles bounding $H(\widetilde{Q_1})\cap H(\widetilde{Q_2})$.
Denoted by $\angle_s \widetilde{P}$ the angel between the sides $\widetilde{P}c_1$ and $\widetilde{P}c_2$.
\par
The following simple observation is useful in the coming proofs.
\begin{lemma}\label{lemmanonsmooth}
Let $\widetilde{P}$ be a non-smooth point of a spherical convex body $\widetilde{C}\subset S^2$.
Let $H(\widetilde{Q})$ and $H(\widetilde{R})$ be extreme supporting hemispheres of 
$\widetilde{C}$ at $\widetilde{P}$.
Then 
\[\angle_s \widetilde{P}+|\widetilde{Q}\widetilde{R}|=\pi.
\] 
\end{lemma}
\begin{figure}[htbp]\label{figure1}
  \includegraphics[width=8cm]{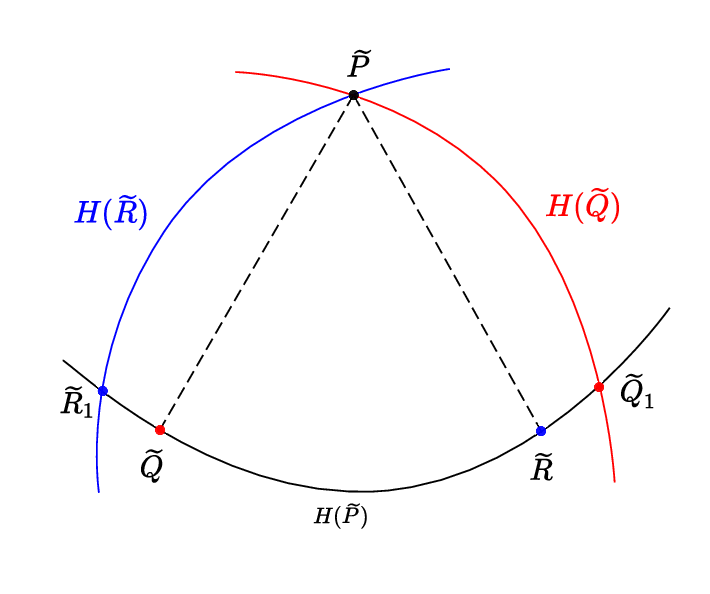}
  \caption{$\angle_s \widetilde{P}+|\widetilde{Q}\widetilde{R}|=\pi.$}
  \label{figure1}
\end{figure}
\begin{proof} 
Since $H(\widetilde{Q}), H(\widetilde{R})$ support $\widetilde{C}$ at $\widetilde{P}$, we know that $\widetilde{Q}, \widetilde{R}$ are two points of $\partial H(\widetilde{P})$.
Thus it is easily follows from following equality (see Figure 1):
\[
\pi=|\widetilde{Q}\widetilde{Q}_1|+|\widetilde{R}\widetilde{R}_1|=|\widetilde{Q}\widetilde{R}|+|\widetilde{Q}_1\widetilde{R}_1|=|\widetilde{Q}\widetilde{R}|+\angle_s \widetilde{P}.
\]
\end{proof}

(1)$\implies$ (2):
Since $P\cdot Q\leq 1$ with equality if and only if $Q=P$, 
we know that $\gamma(P)$ is locally maximum,  
\[
\gamma(P)=\gamma(P)P\cdot P\geq\gamma(P)P\cdot Q=\gamma (Q)
\] for any $Q\in S^n\cap B(P, \delta)$.
Let ${\rm Gr}(n+1, 2)^P$ be any Grassmanian space ${\rm Gr}(n+1, 2)$ of $\mathbb{R}^{n+1}$ contains $\gamma(P)P$. 
Thus it is sufficient to prove that $\gamma(P)P$ is a non-smooth point of 
$\mathcal{W}_\gamma\cap {\rm Gr}(n+1, 2)^P$. 
Identifying ${\rm Gr}(n+1, 2)^P$ with $\mathbb{R}^2$ and 
set 
\[
\mathcal{W}_\gamma^{Gr}= \mathcal{W}_\gamma\cap Gr(n+1, 2)^P
\ \mbox{and}\ 
B_P=B(P, \delta)\cap {\rm Gr}(n+1, 2)^P\subset S^1.
\]
By assumption of (1), we have
\[
\gamma(Q)=\gamma(P)P\cdot Q,
\]
for any point $Q$ of $B_P$. 
This means the inversion of $(Q, \gamma(Q))$ is $(-Q, \frac{1}{\gamma(P)P\cdot Q})$, 
for any point $Q$ of $B_P$.  
Hence, it follows that
\[
\frac{1}{\gamma(P)P\cdot Q}(-Q)\cdot (-P)=\frac{1}{\gamma(P)},
\]
for any $Q\in B_P$.
Therefore,  $\left\{\mbox{inv}(Q, \gamma(Q))\mid Q\in B_P\right\}$ 
is an open subset of the line 
\[
l_{-P}=\{R\in \mathbb{R}^2\mid R\cdot -P=1/\gamma(P)\}.
\]
This means $l_{-p}$ is a supporting line of the convex hull of 
\[  
\tag{$\star$}
\left\{\mbox{inv}(Q, \gamma(Q))\mid Q\in S^n \cap G_{n+1, 2}^P\right\}.
\]
Since 
\[
\left(R-(-P\frac{1}{\gamma(P)})\right)\cdot (-P)=0
\ 
\mbox{for any}\ R\in l_{-P},
\]
we know that $-P$ is a unit normal vector of $l_{-P}$.
{\color{black}
Notice that 
\[
\mbox{inv}(\gamma(P)P)=(-P, \frac{1}{\gamma(P)})
\]
is a point of $l_{-p}$.
By Proposition \ref{psi}, 
\[
\mbox{Id}^{-1}\circ \alpha_N \circ \Psi_N\circ \alpha_N^{-1} \circ{\mbox Id}(\gamma(P)P)=\mbox{{\rm inv}}(\gamma(P)P)\in l_{-p}.
\]
Then it follows that
\[
\tag{$\star\star$}
\Psi_N\circ \alpha_N^{-1} \circ{\mbox Id}(\gamma(P)P)
=
\alpha_N^{-1}\circ\mbox{Id}\circ\mbox{{\rm inv}}(\gamma(P)P)
\in \alpha_N^{-1}\circ\mbox{Id}(l_{-P}).
\]
Since $P$ is a unit normal vector of $l_{-P}$, by definition of $\Psi_N$ and ($\star\star$), 
we have that 
\[
\tag{$\star\star\star$}
\alpha_N^{-1}\circ\mbox{Id}(l_{-P})
\subset 
\partial H( \alpha_N^{-1} \circ{\mbox Id}(\gamma(P)P)).
\] 
Moreover, since 
\[
\{
\mbox{inv}(Q, \gamma(Q))\mid Q\in B_P
\}
\subset
l_{-P},
\]
one knows that 
\[
\alpha_N^{-1}\circ \mbox{Id}\left(\left\{\mbox{inv}(Q, \gamma(Q))\mid Q\in B_P\right\}\right)
\]
is a great circle segment of 
$\partial H( \alpha_N^{-1} \circ{\mbox Id}(\gamma(P)P))$. 
Set 
\[
\mathcal{D}=\alpha_N^{-1}\circ \mbox{Id}\left(\left\{\mbox{inv}(Q, \gamma(Q))\mid Q\in B(P, \delta)\cap G_{n+1, 2}^P\right\}\right)
\ \mbox{and}\ 
\widetilde{\mathcal{W}}_\gamma^{Gr}=\alpha_N^{-1}\circ\mbox{Id}(\mathcal{W}_\gamma^{Gr}).
\]
Then by 
 Proposition \ref{sphericalmethod}, we have that
\[
\mathcal{D}\subset \partial(\widetilde{\mathcal{W}}_\gamma^{Gr})^\circ.
\]
By ($\star$) and ($\star\star\star$), one obtains that $H\left(\alpha_N^{-1} \circ{\mbox Id}(\gamma(P)P)\right)$ is a supporting hemisphere of 
$(\widetilde{\mathcal{W}}_\gamma^{Gr})^\circ$.
Then by Lemma \ref{lemmanonsmooth} and Lemma \ref{boundarysupport}, we have that
$\alpha_N^{-1} \circ{\mbox Id}(\gamma(P)P)$ 
is a non-smooth point of $\partial\widetilde{\mathcal{W}}_\gamma^{Gr}$, 
or equivalently $(P, \gamma(P))$ 
is a non-smooth point of $\partial \mathcal{W}_\gamma^{Gr}$.
By arbitrariness of ${\rm Gr}(n+1, 2)^P$, we know that $(P, \gamma(P))$ is an apex point of $\mathcal{W}_\gamma$.
Since $(P, \gamma(P))=\gamma(P)P$, it is clear that 
$(P, \gamma(P))\in \mathbb{R}_+ P\cap \partial \mathcal{W}_\gamma$.
}
\par
(2)$\implies$ (3):
By definition of apex, $P$ is a non-smooth point of 
\[
\mathcal{W}_\gamma^{Gr}= \mathcal{W}_\gamma\cap Gr(n+1, 2)^P.
\]
By the proof of $(1)$  implies $(2)$, we know that
 $\left\{\mbox{inv}(Q, \gamma(Q))\mid Q\in B(P, \delta)\cap G_{n+1, 2}^P\right\}$ 
is a subset of the line 
\[
l_{-P}=\{R\in \mathbb{R}^2\mid R\cdot -P=1/\gamma(P)\}, 
\]
and contains $(-P, 1/\gamma(P))$ as its relative interior.
By arbitrarily of $Gr(n+1,2)^P$ and compactness of $S^n$, we know that
there exists a positive number $\delta$ such that 
\[
\{(-Q, 1/\gamma(Q))\mid Q\in B(P, \delta)\cap S^n\}
\] is a pice of $\mathcal{H}^n$, where 
$\mathcal{H}^n$ is the affine hyperplane through $\frac{1}{\gamma(P)}(-P),$ and has a normal vector $-P$.
\par
(3)$\implies$ (1): 
Let $Q$ be a point of $B_s^n(P, \delta)=\{Q\in S^n\mid |PQ|<\delta\}$.
By assumption, $(-Q, 1/\gamma(Q))$ be a point of 
the affine hyperplane through $\frac{1}{\gamma(P)}(-P),$ and has a normal vector $-P$. 
Then we have that 
\[
(-Q, \frac{1}{\gamma(Q)})\cdot -P=\frac{1}{\gamma(P)}.
\]
That is 
\[
\frac{\frac{1}{\gamma(P)}}{\frac{1}{\gamma(Q)}}=P\cdot Q.
\]
Thus, the quality $\gamma(Q)=\gamma(P)P\cdot Q$ holds, 
for any $Q$ of $B(P, \delta)\cap S^n$.  \hfill{$\square$}
\par
\bigskip
The  {\it dual Wull shape of } $\mathcal{W}_\gamma$, denoted by $\mathcal{DW}_\gamma$, is the convex hull of 
\[
\{\mbox{inv}(\theta, \gamma(\theta))\mid \theta\in S^n\}.
\] 
As a corollary of Theorem \ref{maintheorem}, we have the
following:
\begin{corollary}
Let $\mathcal{W}_\gamma$ be a Wulff shape. 
The following two assertions are equivalent:
\begin{enumerate}
\item the boundary of $\mathcal{W}_\gamma$ contains an apex point;
\item the boundary of $\mathcal{DW}_\gamma$ contains an $n-$dimensional affine ball $\mathcal{B}$ with center $M$ such that the affine subspace contains $\mathcal{B}$ has normal $M$.
\end{enumerate}
\end{corollary}
\section{Proof of Theorem \ref{theorem2}}
\begin{lemma}\label{lemmaintersection}
Let $\widetilde{C}$ be a spherical convex body of $S^{n+1}$. 
Then there exists a point $\widetilde{P}\in S^{n+1}$ and $\delta>0$ such that
\[
B_s^{n+1}(\widetilde{P}, \delta)\subset (\widetilde{C}\cap \widetilde{C}^\circ).
\]
\end{lemma}
\begin{proof}
Let $B_s^{n+1}(\widetilde{P}, \delta_1)$ and $B_s^{n+1}(\widetilde{P}, \delta_2)$ be the circumscribed cap and inscribed cap of
$\widetilde{C}$ respectively, 
where $0<\delta_1<\delta_2<\pi/2$.
Then by Lemma \ref{lemmainclusion}, it follows that 
\[
B_s^{n+1}(P, \delta_2)^\circ\subset \widetilde{C}^\circ\subset B_s^{n+1}(\widetilde{P}, \delta_1)^\circ.
\]
Set $\delta=\mbox{min}\{\delta_1, \frac{\pi}{2}-\delta_2\}$. 
Thus we have the inclusion $B_s^{n+1}(P, \delta)\subset (\widetilde{C}\cap \widetilde{C}^\circ)$.
\end{proof}
\begin{lemma}[\cite{hwam}]\label{dualconstantwidth} 
Let $\widetilde{C}$ be a spherical convex body in $S^{n+1}$, and $0<\delta<\pi$. The
following two assertions are equivalent:
\begin{enumerate}
\item $\widetilde{C}$ is of constant width $\delta$.
\item $\widetilde{C}^\circ$ is of constant width $\pi-\delta$.
\end{enumerate}
\end{lemma}
\begin{lemma}\label{lemmaboundary}
Let $0<\delta<\pi/2$. 
Let $\widetilde{C}$ be a spherical convex body in $S^2$ such that 
the boundary of $\widetilde{C}$ consists only of arcs of circles of radius $\delta$.
Then the boundary of $\widetilde{C}^\circ$ consists only 
arcs of circles of radius $\pi/2-\delta$ and great circle segments.
\end{lemma}
\begin{proof}
Let  $\mathcal{C}$ be the mapping from  $\partial \widetilde{C}$ to the set consisting of 
subsets of 
$\partial \widetilde{C}^\circ$,  defined by
\[
\mathcal{C}(\widetilde{P})=\{\widetilde{R}\in \partial C^\circ\mid H(\widetilde{R}) \ \mbox{supports}\ \widetilde{C} \ \mbox{at}\ \widetilde{P}\}.
\]
By Lemma \ref{boundarysupport}, 
$\mathcal{C}(\widetilde{P})$ is a subset of $\partial \widetilde{C}^\circ$ for any $\widetilde{P}\in \partial \widetilde{C}$. 
Hence, the mapping $\mathcal{C}$ is well-defined. 
Denote by arc$(\widetilde{P}\widetilde{Q})$ the arc of circle with endpoints  $\widetilde{P}$ and $\widetilde{Q}$.
Let $\widetilde{P}$ be a boundary point of $\widetilde{C}$. 
By the assumption, there exists an arc($\widetilde{Q}\widetilde{R}$) of the boundary of cap $S$ of radius $\delta$ contains $\widetilde{P}$.
Since $S$ is of constant width $2\delta$, it is clear that $S^\circ$ is of constant width $\pi-2\delta$. 
Let $H(\widetilde{Q}_1), H(\widetilde{R}_1)$ be supporting hemispheres of $S$ at $\widetilde{Q}, \widetilde{R}$ respectively. 
\par
If $\widetilde{P}$ is a relative interior point of arc($\widetilde{Q}\widetilde{R}$), then by smoothness of $\widetilde{C}$, 
$\mathcal{C}(P)$ is a relative interior point of the arc($\widetilde{Q}_1\widetilde{R}_1$). 
Then it follows that
\[
\mathcal{C}(\mbox{arc}(\widetilde{Q}\widetilde{R})-\{\widetilde{Q},\widetilde{R}\})=\mbox{arc}(\widetilde{Q}_1\widetilde{R}_1)-\{\widetilde{Q}_1,\widetilde{R}_1\},
\]
is an arc of circle of radius $\pi/2-\delta$.
\par
If $\widetilde{P}$ is an endpoint of arc($\widetilde{Q}\widetilde{R}$), then by assumption, 
there exists an arc($\widetilde{Q}^\prime \widetilde{R}^\prime$) of the boundary of cap $S^\prime$  of radius $\delta$
such that 
\[
\widetilde{P}=\mbox{arc}(\widetilde{Q}\widetilde{R})\cap \mbox{arc}(\widetilde{Q}^\prime \widetilde{R}^\prime).
\]
Then by Lemma \ref{lemmanonsmooth}, it follows that $\mathcal{C}(P)$ is a segment of great circle (see Figure 2). 
\end{proof}

\begin{figure}[htb]
  \includegraphics[clip,width=9.0cm]{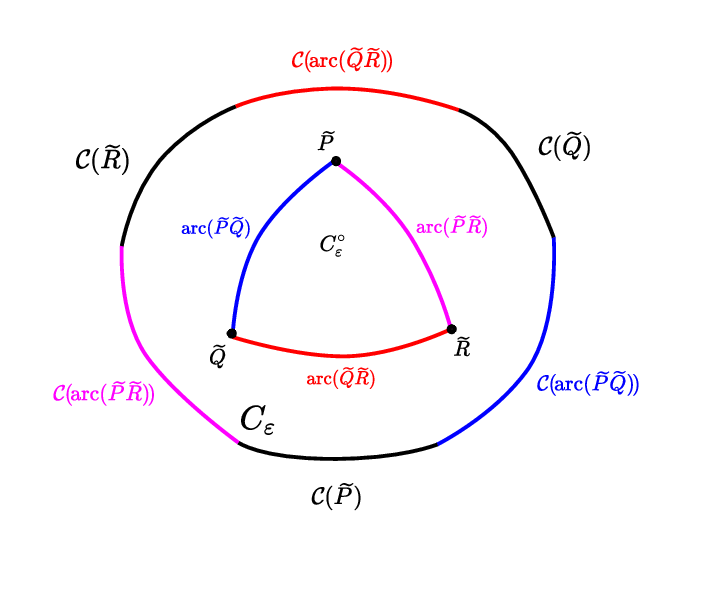}
  \caption{An illustration of corresponding part}
  \label{figure1}
\end{figure}

We are now in position to prove Theorem \ref{theorem2}.
\par
Since $\widetilde{C}$ is of constant width $\tau$,
by Lemma \ref{dualconstantwidth}, we know that $\widetilde{C}^\circ$ is of constant width $\pi-\tau<\pi/2$.
Then by Theorem \ref{theorem3}, for any given $\varepsilon>0$ there exists a convex body 
$\widetilde{C}_\varepsilon^\circ$ of constant width $\pi-\tau$ 
whose boundary consists only arcs of circles of radius $\Delta(\widetilde{C}^\circ)=\pi-\tau$ 
such that 
\[
h(\widetilde{C}^\circ, \widetilde{C}_\varepsilon^\circ)\leq \varepsilon.
\]
By Proposition \ref{dualconvex}, it follows that $\widetilde{C}_\varepsilon=\widetilde{C}_\varepsilon^{\circ\circ}$ is of constant width $\tau$.
Moreover, by Lemma \ref{lemmaintersection} and Proposition \ref{dualisometric}, it follows that
\[
h(\widetilde{C}^\circ, \widetilde{C}_\varepsilon^\circ)=h(\widetilde{C}, \widetilde{C}_\varepsilon).
\]
Therefore, by assumption and Lemma \ref{lemmaboundary}, 
we know that
the boundary of $\widetilde{C}_\varepsilon$ consists only of arcs of circles of radius $\tau-\frac{\pi}{2}$ and great circle segments.
 \hfill{$\square$}
\begin{corollary}
Let $\widetilde{C}$ be a spherical convex body of constant width $\tau>\pi/2$. 
Then there exists a sequence $\{\widetilde{C}_i\}_{i=1}^\infty$ of convex bides 
 of constant width $\tau>\pi/2$, 
whose boundary consists only of arcs of circles of radius $\tau$ and great circle segments 
such that 
\[
\lim_{i\to \infty}\widetilde{C}_i=\widetilde{C}
\]
with respect to the Hausdorff distance.
\end{corollary}
We conjecture that any spherical convex body of constant width $\pi/2$ can be approximated by a sequence of spherical polytope of constant width $\pi/2$. 
\par
Any convex bodies of constant width $\tau\neq \pi/2$ can not be approximated by 
spherical polytope of constant width $\tau$. 
This is because if a spherical convex polytope $P$ is of constant width $\tau$, then $\tau=\pi/2$(\cite{han1}).

\par
\bigskip
{\bf Acknowledgements.}
This work was supported, in partial, by
Natural Science Basic Research Plan in Shaanxi Province of China
(2023-JC-YB-070).

\end{document}